\documentclass[12pt]{amsart} 
\usepackage{amsmath,amssymb,amsthm, amsfonts, amssymb, color, enumerate}
\usepackage[dvipsnames]{xcolor}
\usepackage{fullpage}
\usepackage{mathtools}
\usepackage{amsfonts}
\usepackage{tikz-cd}
\usepackage{graphicx}
\usepackage{kantlipsum}
\usepackage{mathtools}
\usepackage{bbm}
\usepackage{xcolor}
\definecolor{vry}{RGB}{253, 231, 37}

\definecolor{vrg}{RGB}{94,201,98}

\definecolor{vrdg}{RGB}{33, 145, 140}

\definecolor{vrb}{RGB}{59,82,139}

\definecolor{vrp}{RGB}{68,1,84}

\definecolor{vro}{RGB}{249,142,9}

\definecolor{vrr}{RGB}{188,55,84}

\definecolor{vrnb}{RGB}{13,8,135}

\usepackage[colorlinks=true, linkcolor=vrr, citecolor=vrnb, filecolor=vrr, urlcolor=vrr]{hyperref}

\theoremstyle{plain}%
\newtheorem{theorem}{Theorem}[section]
\newtheorem{lemma}[theorem]{Lemma}

\newtheorem{claim}[theorem]{Claim}%

\newtheorem{proposition}[theorem]{Proposition}

\theoremstyle{plain}%
\newtheorem{remark}[theorem]{Remark}%

\newtheorem{assumption}[theorem]{Assumption}%

\usepackage{enumerate}  
\DeclareMathOperator{\per}{per}

\newcommand{\R}{\mathbb R}
\newcommand{\one}{\mathbf 1}
\newcommand{\detF}{\det\nolimits_F}
\newcommand{\ip}[2]{\left\langle #1,#2\right\rangle}
\newcommand{\norm}[1]{\left\lVert #1\right\rVert}
\newcommand{\HS}{\mathrm{HS}}
\newcommand{\tr}{\operatorname{tr}}
\newcommand{\Sym}[1]{\mathfrak{S}_{#1}}

\title[Partition function of Mallows model]{On the partition function of a class of Mallows model}
\author[R.~Tripathi]{Raghavendra~Tripathi}
\address{Raghavendra Tripathi, Division of Science, New York University, Abu Dhabi, UAE}
\email{rt1986@nyu.edu}
\date{}
\thanks{I would like to thank Soumik Pal for suggesting this problem, for carefully reading the first draft, and for making several useful suggestions. I would also like to thank Manjunath Krishnapur for his encouragement and discussions on this problem.}
\begin{document}
\maketitle


\begin{abstract}
Let $\Sym{n}$ denote the set of all permutations on $n$ labels. Let $c:[0, 1]^2\to [0, \infty)$ be a twice continuously differentiable function. A subfamily of the Mallows model is the Gibbs probability measures on $\Sym{n}$ such that $\mathbb{P}(X=\sigma)=L_n^{-1} \prod_{i=1}^{n}\exp(-c(i/n, \sigma(i)/n))$. Mukherjee [Ann. Stat., Vol. 44(2), pp 853--875 (2016)] computed the limit of the log partition function and showed that $\lim_{n\to \infty}\frac{1}{n}\log L_n=-\Gamma_0$ where $\Gamma_0$ is the optimal cost associated with an entropy regularized optimal transport problem. In the KRP Memorial Volume of the Indian Journal of Pure and Applied Math, Pal conjectured an exact value for the limit $\lim_{n\to \infty} e^{-n\Gamma_0}L_n$ in terms of the Fredholm determinant of an integral operator and provided a partial proof. We give a complete proof of Pal's conjecture.
\end{abstract}



\section{Introduction}
\label{sec:intro}
Let $n\in \mathbb{N}$ and let $\Sym{n}$ denote the set of all permutations on $n$ labels. Mallows model is a popular measure on $\Sym{n}$ introduced by~\cite{mallows1957non} as a model for the highly ordered random permutation. Mallows model arises as the stationary distribution of biased shufflings~\cite{diaconis2000analysis,benjamini2005mixing}. Mallows model has been extensively studied and found several applications. We refer the reader to the PhD thesis~\cite{levy2017novel} for a nice survey of works related to the Mallows model and some recent applications of the Mallows model in probability and combinatorics. In general, the Mallows model is a Gibbs probability measure where the Hamiltonian is given by a right-invariant divergence, for instance, Spearman's foot rule, or Spearman's rank correlation~\cite{Diaconis88Group}. We consider the following exponential family of measures, which is a subclass of the Mallows model.

\begin{assumption}
\label{assmp:CostFunction}
   Let $c:[0, 1]^2\to [0, \infty)$ be a function satisfying the following conditions: 
    \begin{enumerate}
    \item $c(x, x)=0$
    \item $c(x, y)=c(y, x)$
    \item $c(1-x, 1-y)=c(x, y)$
    \item $c$ is twice continuously differentiable.
\end{enumerate}
\end{assumption}

Consider the Gibbs probability measure on $\Sym{n}$ given by 
\[
\mathbb{P}(X=\sigma) = L_n^{-1} \prod_{i=1}^{n} \exp\left(-c(i/n, \sigma(i)/n)\right) \;,\]
where $L_n$ is the partition function given by 
\begin{equation}
    \label{eq:PartitionDefin}
    L_n = \frac{1}{n!}\sum_{\sigma\in \Sym{n}} \prod_{i=1}^{n} \exp\left(-c(i/n, \sigma(i)/n)\right)\;.
\end{equation}

Motivated by the estimation problems for Mallows model, Mukherjee~\cite{mukherjee2016estimation} studied the asymptotics of the partition function $L_n$ of the above measure and showed that 
\[\lim_{n\to \infty}\frac{1}{n}\log(L_n)=-\Gamma_0\;,\]
for some constant $\Gamma_0$ which was computed explicitly by Mukherjee. We also refer the reader to~\cite{ali25mallows} for a more recent study of the estimation problem in the context of the Mallows model. 

Pal~\cite{pal2024limiting} observed that $\Gamma_0$ is the value of an entropy regularized optimal transport with uniform marginals. For two probability measures $\pi_0$ and $\pi_1$, let $\Pi(\pi_0, \pi_1)$ denote the set of all couplings of $\pi_0$ and $\pi_1$. And, let $\lambda$ denote the Lebesgue measure on $[0, 1]$. Then, the quantity $\Gamma_0$ can be defined as 
\[ \Gamma_0 := \inf_{\xi\in \Pi(\lambda, \lambda)}\left[\int c(x, y)\xi(x, y)\lambda(dx)\lambda(dy)+ \operatorname{Ent}(\xi)\;,\right]\]
where $\operatorname{Ent}(\cdot)$ is the so-called entropy function defined as 
\[\operatorname{Ent}(\xi) = \int \xi(x, y)\log(\xi(x, y))\lambda(dx)\lambda(dy),\] if $\xi$ is absolutely continuous with rspect to the product measure $\lambda(dx)\lambda(dy)$ and $+\infty$ otherwise. By a standard abuse of notation, we use the same symbol $\xi$ for a measure as well as its density with respect to the Lebesgue measure on $[0, 1]^2$. We refer the reader to~\cite{Leo12Schro} for the connection between the entropy regularized optimal transport and the Schr\"odinger bridge, and to~\cite{villani2021topics} for a general introduction to the optimal transport. It is known~\cite{RT93Note, Csi75} that the minimizer for the above problem is a measure $\rho\in \Pi(\lambda, \lambda)$ that has a density (with respect to the Lebesgue measure on $[0, 1]^2$) given by 
\begin{equation}
\label{eqn:density}
    \rho(x, y) = \exp\left(-c(x, y)-a(x)-a(y)\right)\;,
\end{equation}
for some measurable function $a$ that satisfies the marginal constraints given by 
\begin{equation}
\label{eqn:Marginal}
    \exp(-a(x)) = \int \exp(-c(x, y)-a(y))\lambda(dy)\;,\qquad x\in [0, 1]\;.
\end{equation}
It follows that 
\begin{align*}
    \Gamma_0 &= \int c(x, y)\rho(x, y)\lambda(dx)\lambda(dy)+ \operatorname{Ent}(\rho)\\
    &= \int c(x, y)\rho(x, y)\lambda(dx)\lambda(dy)-\int (c(x, y)+a(x)+a(y))\rho(x, y)\lambda(dx)\lambda(dy)\\
    &= -2\int a(x)\rho(x, y)\lambda(dx)\lambda(dy) = -2\int a(x)\lambda(dx)\;.
\end{align*}
If $c$ is twice continuously differentiable, it can be shown using~\eqref{eqn:density} and~\eqref{eqn:Marginal} that $a$ (and therefore $\rho$) must be twice continuously differentiable.  And, therefore, 
\[ \frac{1}{n}\sum_{i=1}^{n}a(i/n) -\int_{0}^{1}a(x)dx =\frac{2}{n^2}\sum_{i=1}^{n}a'(i/n)+O(n^{-2})\;.\]
This implies that \[ \sum_{i=1}^{n}a(i/n) =n\Gamma_0 +O(1/n)\;.\]
Using this, we conclude that 
\begin{align*}
    D_n &=\frac{1}{n!} \sum_{\sigma\in \Sym{n}}\prod_{i=1}^{n}\rho(i/n, \sigma(i)/n)\\
    &= L_n \exp(-2\sum_{i=1}^{n}a(i/n))\\
    &= L_n\exp(n\Gamma_0)e^{O(1/n)} = (1+o_n(1)) L_n\exp(n\Gamma_0)\;.
\end{align*}
This means that 
\begin{equation}
    \lim_{n\to \infty}D_n= \lim_{n\to \infty}L_n\exp(n\Gamma_0),
\end{equation}
provided the limits above exist. Motivated by this, Pal~\cite{pal2024limiting} conjectured that this limit exists and gave an explicit description of the limit of $D_n$ in terms of the Fredholm determinant of an integral operator. To describe the limit of $D_n$, we first need some setup.

\subsection{Setup and the main result}
Let $\rho\colon [0,1]^2\to (0,\infty)$ be the density of optimal coupling in the entropy regularized optimal transport problem as in~\ref{eqn:density}. Recall that $\rho$ is twice continuously differentiable. We think of $\rho$ as a kernel, and since the marginals of $\rho$ are uniform, $\rho$ is a doubly stochastic kernel. That is, 
\begin{equation}\label{eq:marginals}
\int_0^1 \rho(x,y)\,\lambda(dy) = 1
\qquad\text{and}\qquad
\int_0^1 \rho(x,y)\,\lambda(dx) = 1
\end{equation}
for every $x,y\in[0,1]$.
Let $T$ be the self-adjoint Hilbert-Schmidt operator on $L^2([0,1])$ associated with the kernel $\rho$, that is,
\[
(Tf)(x) := \int_0^1 \rho(x,y)f(y)\,\lambda(dy), \qquad f\in L^2[0, 1]\;.
\]
Let
\[
H := \Bigl\{f\in L^2([0,1]) : \int_0^1 f(x)\,dx = 0\Bigr\}.
\]
\begin{assumption}[Spectral gap]\label{ass:gap}
Assume that
\begin{equation}\label{eq:gap}
\lambda_* := \norm{T\vert_H}_{L^2\to L^2} < 1.
\end{equation}
\end{assumption}
Let $\one\in L^2[0, 1]$ denote the constant function takinng values $1$ almost everywhere. Since $T\one=\one$ by \eqref{eq:marginals}, the subspace $H$ is $T$-invariant. Since $T\vert_H$ is Hilbert-Schmidt, $(T\vert_H)^2$ is trace class, and the Fredholm determinant of $(I-(T\vert_H)^2)$ is well-defined, and it is given by the convergent infinite product 
\[
\detF(I-(T\vert_H)^2) = \prod_{n=1}^{\infty}(1-\lambda_n^2)\;,
\]
where $(\lambda_i)_i$s are eigenvalues of $T\vert_H$ (see~\cite[Chapter 14]{Neer22}). We now state the main result, which was conjectured by Pal~\cite{pal2024limiting}. 
Let $R_n$ be $n\times n$ matrix $R_n:=\left(\frac{1}{n}\rho(i/n, j/n)\right)_{1\le i,j\le n}$.
Then,
\begin{equation}
\label{eqn:DefinitionDn}
    D_n = \frac{1}{n!} \sum_{\sigma\in \Sym{n}}\prod_{i=1}^{n}\rho(i/n, \sigma(i)/n) = \frac{1}{n!}\,\per\!\left(nR_n\right),
\end{equation}
where $\per(A)$ denotes the permanent of the matrix $A$.

\begin{theorem}[Pal's Conjecture]
\label{thm:main}
Let $\rho:[0, 1]^2\to [0, \infty)$ be a twice continuously differentiable function satisfying Assumption~\ref{ass:gap}. Let $D_n$ be as defined in~\eqref{eqn:DefinitionDn}. Then, 
\[
\lim_{n\to\infty} D_n = \detF\bigl(I-(T\vert_H)^2\bigr)^{-1/2}.
\]
\end{theorem}

\begin{remark}
Our proof does not actually use the conditions $c(x,x)=0$ or $c(1-x,1-y)=c(x,y)$ on the cost function $c$. 
\end{remark}

\section{Proof of Theorem~\ref{thm:main}}
In this section, we prove Theorem~\ref{thm:main} assuming some Lemmas that we prove in Section~\ref{sec:remaining}. For notational convenience, we will write $x_{i, n}=i/n$. We first observe that even though the kernel $\rho$ is doubly stochastic, the matrix $R_n:=(\frac{1}{n}\rho(x_{i, n}, x_{j, n}))_{1\leq i, j\leq n}$ is not doubly stochastic. However, we first show that $R_n$ is almost doubly stochastic.
\begin{lemma}[Doubly stochastic perturbation]
\label{lem:DSPerturbation}
Fix $n\in \mathbb{N}$. There exists a vector $h_n\in \mathbb{R}^n$ ssatisfying:
\[
\norm{h_n}_{2, n} = O(n^{-1}),\quad \norm{h_n}_{\infty} = O(n^{-1/2}), \quad \text{and} \quad  \sum_{i=1}^{n}\log (1+h_n(i))=O(n^{-1})\;.
\]
Let $u^{(n)}:=\one_n+h_n$, where $\one_n\in \mathbb{R}^n$ be the vector all whose coordinates are $1$. Define
\[
\widehat\rho^{(n)}_{ij} := u_i^{(n)}\,\rho(x_{i,n},x_{j,n})\,u_j^{(n)}, \qquad 1\leq i, j\leq n\;.
\]
Then, the matrix $A_n:=\left(\frac{1}{n}\widehat\rho^{(n)}\right)$ is symmetric and doubly stochsstic. 
\end{lemma}
The following is an immediate consequence of Lemma~\ref{lem:DSPerturbation}.
\begin{lemma}[Comparison of $D_n$ and $\widehat D_n$]
\label{lem:Dcompare}
Define 
\[
\widehat D_n := \frac1{n!}\,\per\bigl(\widehat\rho^{(n)}_{ij}\bigr)_{1\le i,j\le n}.
\]
Then, 
\[
\widehat D_n = D_n\bigl(1+O(n^{-1})\bigr).
\]
\end{lemma}
\begin{proof}
Multiplying row $i$ and column $i$ of a matrix by the same scalar $u_i^{(n)}$ multiplies the permanent by $\prod_{i=1}^n (u_i^{(n)})^2$. Hence
\[
\widehat D_n
= D_n\prod_{i=1}^n (u_i^{(n)})^2
= D_n\exp\!\left(2\sum_{i=1}^n \log\bigl(1+(h_n)_i\bigr)\right).
\]
The conclusion follows from the fact $\sum_{i=1}^{n}\log (1+h_n(i))=O(n^{-1})$ in Lemma~\ref{lem:DSPerturbation}.

\end{proof}
To conclude the proof, we need the following asymptotic for the permanent of a scaled doubly stochastic matrix~\cite[Section 3]{mccullagh2014asymptotic} in the so-called moderate deviation regime. 
\begin{proposition}[McCullagh asymptotic]\label{prop:mcc}
For each $n$, let $A_n$ be a doubly stochastic matrix of order $n$, and let $J_n$ denote the matrix with all entries equal to $1/n$. Assume that:
\begin{enumerate}
    \item for every $p\ge 1$,
    \[
    \sup_{n\ge 1}\,\frac{1}{n^2}\sum_{i,j=1}^n \bigl|n(A_n-J_n)_{ij}\bigr|^p < \infty;
    \]
    \item there exists $\delta>0$ such that every non-trivial eigenvalue $\lambda$ of $A_n$ satisfies $|\lambda|\le 1-\delta$.
\end{enumerate}
Then,
\begin{equation}
\label{eq:mcc}
\frac{1}{n!}\,\per(nA_n)
= \det\bigl(I+J_n-A_n^{\top}A_n\bigr)^{-1/2}\bigl(1+O(n^{-1})\bigr).
\end{equation}
\end{proposition}

\noindent We will apply Proposition~\ref{prop:mcc} to the doubly stochastic matrix matrix $A_n:=\left(\frac{1}{n}\widehat\rho^{(n)}\right)$. To this end, we need to check that $A_n$ indeed satisfies the assumptions of Proposition~\ref{prop:mcc}. We record this as a Lemma below, the proof is deferred to the Section~\ref{sec:remaining}.
\begin{lemma}\label{lem:A_nsatisfiesPropmcc}
   The matrix $A_n:=\left(\frac{1}{n}\widehat\rho^{(n)}\right)$ satisfies the assumptions of Proposition~\ref{prop:mcc}.
\end{lemma} 
Let $B_n:=A_n-J_n$. It is easy to verify that $B_nJ_n=J_nB_n=0$. In particular, $I+J_n-A_n^2 = I-B_n^2$. Therefore, Lemma~\ref{lem:A_nsatisfiesPropmcc} and Proposition~\ref{prop:mcc} yield
\[
\widehat D_n = \det(I - B_n^2)^{-1/2}(1+O(n^{-1}))\;.
\]
The proof of Theorem~\ref{thm:main} is complete with the following Proposition. 
\begin{proposition}
\label{prop:FinalNail}
Let $B_n$ be as above. Then, 
\[
\lim_{n\to \infty}\det(I-B_n^2)  = \det_F(I-(T\vert_H)^2)\;.
\]
\end{proposition}

\section{Remaining proofs}
\label{sec:remaining}
\subsection{Towards the proof of Lemma~\ref{lem:DSPerturbation}}
We begin with a simple observation. For $f\in C^2([0,1])$ define the right-endpoint Riemann sum
\[
\mathcal{R}_n(f):=\frac1n\sum_{j=1}^n f\!\left(\frac{j}{n}\right).
\]
The following Lemma quantifies the error between the Riemann sum and the definite integral. The proof is a straightforward calculus exercise; we skip the proof
\begin{lemma}\label{lem:em}
For every $f\in C^2([0,1])$,
\[
\mathcal{R}_n(f)=\int_0^1 f(t)\,dt+\frac{f(1)-f(0)}{2n}+O\!\left(\frac1{n^2}\right),
\]
where the implicit constant depends only on $\norm{f''}_{\infty}$.
\end{lemma}

We first begin with a lemma that quantifies the idea $R_n$ is almost a doubly stochastic matrix. We begin with some notations. Equip $\R^n$ with the normalized inner product $\ip{u}{v}_n:=\frac1n\sum_{i=1}^n u_iv_i$ and $\norm{u}_{2,n}^2:=\ip{u}{u}_n$.

\begin{lemma}\label{lem:q}
Let $R_n= (\frac1n\rho(x_{i,n},x_{j,n}))_{1\leq i, j\leq n}$ be $n\times n$ matrix as above. Let $\one_n=(1,\dots,1)\in\R^n$.  Define 
\[
q_n:=R_n\one_n-\one_n,
\qquad
\bar q_n:=\frac1n\sum_{i=1}^n (q_n)_i\;.
\]
Then, 
\[
\norm{q_n}_{\infty}=O(n^{-1}),
\quad
\norm{q_n}_{2,n}=O(n^{-1}), \quad |\bar q_n|=O(n^{-2})\;.
\]
\end{lemma}
\begin{proof}
Fix $i$ and apply Lemma \ref{lem:em} to the function $\rho(x_{i,n},y)$. Since $\rho\in C^2([0,1]^2)$, we get
\[
\frac1n\sum_{j=1}^n \rho(x_{i,n},x_{j,n})
=\int_0^1 \rho(x_{i,n},y)\,dy+\frac{\rho(x_{i,n},1)-\rho(x_{i,n},0)}{2n}+O(n^{-2}),
\]
where the implicit constant is independent of $i$. 
Since $\int_0^1\rho(x_{i,n},y)\,dy=1$, we get
\[
(q_n)_i=\frac{\rho(x_{i,n},1)-\rho(x_{i,n},0)}{2n}+O(n^{-2}),
\]
uniformly in $i$. This proves $\norm{q_n}_{\infty}=O(n^{-1})$, and then
$\norm{q_n}_{2,n}\le \norm{q_n}_{\infty}=O(n^{-1})$.

Similarly, applying Lemma~\ref{lem:em} to the $C^2$ function $g(x_{i, n}):=\rho(x_{i, n},1)-\rho(x_{i, n},0)$, and observing that $\int_0^1 g(x)\,dx =0$, we obtain
\[
\bar q_n
=\frac{1}{2n}\cdot \frac1n\sum_{i=1}^n g(x_{i,n})+O(n^{-2})
=\frac{1}{2n}\left(\int_0^1 g(x)\,dx+O(n^{-1})\right)+O(n^{-2})
=O(n^{-2}).
\]

\end{proof}

Before we prove Lemma~\ref{lem:DSPerturbation}, we make some more observations. Let $V_n\subset L^2([0,1])$ denote the subspace of functions that are constant on each interval
$I_{i,n}:=((i-1)/n,i/n]$. Then the map
\[
\iota_n\colon (\R^n,\norm{\cdot}_{2,n})\to V_n,
\qquad
(\iota_n u)(x):=u_i \text{ for } x\in I_{i,n},
\]
is an isometric isomorphism. Let $S_n$ be the integral operator on $L^2[0, 1]$ associated with the piecewise constant kernel $\rho_n(x,y):=\rho(x_{i,n},x_{j,n})$ on  $ I_{i,n}\times I_{j,n}$.
Then, $S_n\iota_n=\iota_n R_n$. In particular,
\[
\norm{R_n}_{\R^n\to\R^n}=\norm{S_n|_{V_n}}_{L^2\to L^2}.
\]

\begin{lemma}\label{lem:Sn}
As $n\to\infty$,
\[
\norm{S_n-T}_{\HS}\to 0,
\qquad
\norm{S_n-T}_{L^2\to L^2}\to 0.
\]
Moreover, $I+S_n$ and $I+R_n$ are invertible for all sufficiently large $n$, and there exists a constant
$C_0>0$ such that
\[
\norm{(I+R_n)^{-1}}_{\mathbb{R}^n\to \mathbb{R}^n}=\norm{(I+S_n)^{-1}}_{L^2\to L^2}\le C_0
\]
for all sufficiently large $n$.
\end{lemma}

\begin{proof}
Since $\rho\in C([0,1]^2)$, the piecewise-constant kernels $\rho_n$ converge uniformly to $\rho$, and therefore also in $L^2([0,1]^2)$. Hence $\norm{S_n-T}_{\HS}=\norm{\rho_n-\rho}_{L^2([0,1]^2)}\to 0$.  The operator norm convergence follows from the fact that $\norm{A}_{L^2\to L^2}\le \norm{A}_{\HS}$. Recall that $T\one=\one$ and $T|_H$ is self-adjoint with norm $\lambda_*<1$ by Assumption~\ref{ass:gap}. It follows that the spectrum of $T$ is
contained in $[-\lambda_*,\lambda_*]\cup\{1\}$. Hence, $I+T$ is invertible and
\[
\norm{(I+T)^{-1}}_{L^2\to L^2}\le \frac{1}{1-\lambda_*}.
\]
Choose $n$ sufficiently large, so that
\[
\norm{S_n-T}_{L^2\to L^2}\le \frac{1-\lambda_*}{2}.
\]
In particular, $\norm{(I+T)^{-1}(S_n-T)}_{L^2\to L^2}\leq \frac{1}{2}$. Hence, $(I+(I+T)^{-1}(S_n-T))$ is invertible, and the norm of the inverse is uniformly bounded in $n$. It follows that
\[
I+S_n=(I+T)\Bigl(I+(I+T)^{-1}(S_n-T)\Bigr),
\]
is invertible with uniformly bounded norm. 

Finally, observe that $V_n$ is invariant under $S_n$ and $R_n$ is the matrix of operator $S_n\vert_{V_n}:V_n\to V_n$ with respect to the orthonormal basis $\{e_{i, n}: 1\leq i\leq n\}$ where $e_{i, n}(j)=\sqrt{n}\delta_{i=j}$. It follows that $(I+R_n)$ is invertible for sufficiently large $n$ and that $\norm{(I+R_n)^{-1}}_{\mathbb{R}^n\to \mathbb{R}^n}\leq C_0$ for $n$ sufficiently large $n$.   

\end{proof}

\begin{proof}[Proof of Lemma~\ref{lem:DSPerturbation}]

Write $u^{(n)}=\one_n+h$. Observe that the condition $A_n$ is doubly stochastic is equivalent to $(1+h_i)\left(1+(q_n)_i+(R_nh)_i\right)=1$ for all $1\leq i\leq n$. This is, in turn, equivalent to
\begin{equation}\label{eq:fp}
(I+R_n)h=-q_n-h\circ q_n-h\circ(R_nh),
\end{equation}
where $\circ$ denotes entrywise multiplication. Let us define $\Psi_n:\mathbb{R}^n\to \mathbb{R}^n$ as
\[
\Psi_n(h):=-(I+R_n)^{-1}\bigl(q_n+h\circ q_n+h\circ(R_nh)\bigr).
\]

\begin{claim}[$\Psi_n$ has a unique fixed point] 
\label{claim:FixedPoint}
Let $B_n(M):=\{h\in\R^n:\norm{h}_{2,n}\le M/n\}$. There exists $M>0$ such that $h=\Psi_n(h)$ has a unique solution in $B_n(M)$ for every $n$ sufficiently large.
\end{claim}
\begin{proof}
We show that with appropiate choice of $M$, we have $\Psi_n(B_n(M))\subseteq B_n(M)$ and that $\Psi_n:B_n(M)\to B_n(M)$ is a contraction. The claim then follows from a standard fixed-point argument. To this end, we begin with the observation that if $h\in B_n(M)$, then Lemma \ref{lem:q} gives
\[
\norm{h\circ q_n}_{2,n}\le \norm{q_n}_{\infty}\norm{h}_{2,n}\le \frac{CM}{n^2}.
\]
By Lemma \ref{lem:Sn}, the operators $R_n$ are uniformly bounded on $(\R^n,\norm{\cdot}_{2,n})$. Let $C_1$ be such that $\sup_{n}\norm{R_n}\le C_1$. Since $\norm{h}_{\infty}\le \sqrt n\,\norm{h}_{2,n}\le \frac{M}{\sqrt n}$, we also get
\[
\norm{h\circ(R_nh)}_{2,n}
\le \norm{h}_{\infty}\,\norm{R_nh}_{2,n}
\le \sqrt n\,\norm{h}_{2,n}\,C_1\norm{h}_{2,n}
\le \frac{C_1M^2}{n^{3/2}}.
\]
Since $\norm{q_n}_{2,n}=O(n^{-1})$, we obtain
\[
\norm{\Psi_n(h)}_{2,n}
\le C_0\left(\frac{C}{n}+\frac{CM}{n^2}+\frac{C_1M^2}{n^{3/2}}\right).
\]
Choose $M$ large and then $n$ large so that $\Psi_n(B_n(M))\subset B_n(M)$. We now show that $\Psi_n$ is a contraction on $B_n(M)$ with $M$ as above. To this end, fix $h,g\in B_n(M)$. Note that 
\[
\norm{h\circ q_n-g\circ q_n}_{2,n}
\le \norm{q_n}_{\infty}\,\norm{h-g}_{2,n}
\le \frac{C}{n}\norm{h-g}_{2,n}.
\]
Moreover,
\begin{align*}
\norm{h\circ(R_nh)-g\circ(R_ng)}_{2,n}
&\le \norm{h\circ R_n(h-g)}_{2,n}+\norm{(h-g)\circ R_ng}_{2,n}\\
&\le \norm{h}_{\infty}\,\norm{R_n(h-g)}_{2,n}+\norm{h-g}_{2,n}\,\norm{R_ng}_{\infty}\\
&\le C_1\sqrt n\,(\norm{h}_{2,n}+\norm{g}_{2,n})\,\norm{h-g}_{2,n}\\
&\le \frac{2C_1M}{\sqrt n}\norm{h-g}_{2,n}.
\end{align*}
Hence
\[
\norm{\Psi_n(h)-\Psi_n(g)}_{2,n}
\le C_0\left(\frac{C}{n}+\frac{2C_1M}{\sqrt n}\right)\norm{h-g}_{2,n}\;.
\]
As the constants $C, C_0, C_1, M$ are indepenent of $n$, we conclude that 
\[\norm{h\circ(R_nh)-g\circ(R_ng)}_{2,n}\leq \frac{1}{2}\norm{h-g}_{2, n},\]
for all $n$ sufficiently large. Therefore, $\Psi_n$ has a unique fixed point $h_n\in B_n(M)$. 

\end{proof}

Note that $\norm{h_n}_{2,n}=O(n^{-1})$ and $\norm{h_n}_{\infty}\le \sqrt n\,\norm{h_n}_{2,n}=O(n^{-1/2})$ follow from the fact that $h_n\in B_n(M)$. Finally observe that $1+h_n(i)>0$ for all large $n$, so $u^{(n)}:=\one_n+h_n$ is positive and
\[
u^{(n)}_i\bigl(R_nu^{(n)}\bigr)_i=1,
\qquad 1\le i\le n.
\]
Consequently, $A_n$ is symmetric and doubly stochastic. 

It remains to prove the logarithmic bound $ \sum_{i=1}^{n}\log (1+h_n(i))=O(n^{-1})$. To this end, put
\[
m_n:=\frac1n\sum_{i=1}^n h_n(i)=\ip{h_n}{\one_n}_n.
\]
Taking the $\ip{\cdot}{\one_n}_n$ inner product in~\eqref{eq:fp} with $h=h_n$, we get 
\[
\ip{(I+R_n)h_n}{\one_n}_n
=-\bar q_n-\ip{h_n\circ q_n}{\one_n}_n-\ip{h_n\circ(R_nh_n)}{\one_n}_n.
\]
Since $R_n$ is symmetric and $R_n\one_n=\one_n+q_n$,
\[
\ip{(I+R_n)h_n}{\one_n}_n
=\ip{h_n}{\one_n}_n+\ip{h_n}{R_n\one_n}_n
=2m_n+\ip{h_n}{q_n}_n.
\]
Also,
\[
\ip{h_n\circ q_n}{\one_n}_n=\ip{h_n}{q_n}_n,
\qquad
\ip{h_n\circ(R_nh_n)}{\one_n}_n=\ip{h_n}{R_nh_n}_n.
\]
Therefore
\[
2m_n=-\bar q_n-2\ip{h_n}{q_n}_n-\ip{h_n}{R_nh_n}_n.
\]
Using Lemma~\ref{lem:q} and Lemma~\ref{lem:Sn} together with the bound $\norm{h_n}_{2, n}=O(n^{-2})$, we obtain
\[
|m_n|
\le Cn^{-2}+C\norm{h_n}_{2,n}\norm{q_n}_{2,n}+C\norm{h_n}_{2,n}^2
=O(n^{-2}).
\]
Since $\norm{h_n}_{\infty}\to 0$, the expansion $\log(1+t)=t+O(t^2)$ is uniform for
$|t|\le \frac12$. We conclude that, for all $n$ sufficiently large, 
\[
\sum_{i=1}^n \log(1+h_n(i))
=\sum_{i=1}^n h_n(i)+O\!\left(\sum_{i=1}^n h_n(i)^2\right)
=nm_n+O\!\left(n\norm{h_n}_{2,n}^2\right)
=O(n^{-1}).
\]

\end{proof}

\subsection{Towards the proof of Lemma~\ref{lem:A_nsatisfiesPropmcc}}
Let $\widehat\rho_n$ be the piecewise-constant kernel on $[0,1]^2$ defined by
\[
\widehat\rho_n(x,y) := \widehat\rho^{(n)}_{ij}
\qquad\text{for }(x,y)\in I_{i,n}\times I_{j,n},
\]
and define the piecewise-constant zero-marginal kernel $k_n(x,y) := \widehat\rho_n(x,y)-1$. Let $K_n$ be the integral operator on $L^2([0,1])$ with kernel $k_n$.

\begin{lemma}
\label{lem:Knconv}
As $n\to\infty$,
\begin{equation}\label{eq:Knconv}
\norm{K_n-K}_{\HS}\to 0.
\end{equation}
Consequently,
\begin{equation}\label{eq:Knspectral}
\norm{K_n\vert_H}_{L^2\to L^2} \le \lambda_* + o(1).
\end{equation}
\end{lemma}

\begin{proof}
Since $\rho_n\to\rho$ in $L^2([0,1]^2)$. It therefore suffices to show that $\widehat\rho_n-\rho_n\to 0$ in $L^2([0,1]^2)$. For $(x,y)\in I_{i,n}\times I_{j,n}$,
\begin{align*}
    |\widehat\rho_n(x,y)-\rho_n(x,y)| &= |\rho(x_{i,n},x_{j,n})\Bigl((1+(h_n)_i)(1+(h_n)_j)-1\Bigr)|\\
    &\leq C |(h_n)_i+(h_n)_j+(h_n)_i(h_n)_j|,
\end{align*}
where $C$ is some constant such that $\norm{\rho}_{\infty}\leq C$. In the following discussion, the constant $C$ can change from line to line. Note that 
\begin{align*}
\norm{\widehat\rho_n-\rho_n}_{L^2([0,1]^2)}^2
&\le \frac{C}{n^2}\sum_{i,j=1}^n \bigl((h_n)_i^2+(h_n)_j^2+(h_n)_i^2(h_n)_j^2\bigr) \\
&\le C\left(\frac1n\sum_{i=1}^n (h_n)_i^2 + \Bigl(\frac1n\sum_{i=1}^n (h_n)_i^2\Bigr)^2\right) = O(n^{-2})
\end{align*}
by Lemma~\eqref{lem:DSPerturbation}. This proves \eqref{eq:Knconv}. The operator norm estimate \eqref{eq:Knspectral} follows from
\[
\norm{K_n\vert_H}_{L^2\to L^2}
\le \norm{K\vert_H}_{L^2\to L^2} + \norm{K_n-K}_{L^2\to L^2}
\le \lambda_* + \norm{K_n-K}_{\HS}.
\]
\end{proof}

\begin{proof}[Proof of Lemma~\ref{lem:A_nsatisfiesPropmcc}]
First note that $A_n$ is doubly stochastic by construction. Now observe that the
entries $\widehat\rho^{(n)}_{ij}$ are uniformly bounded because $\rho$ is bounded and $u_i^{(n)}=1+O(n^{-1/2})$ uniformly in $i$. And, therefore, the entries 
\[
n(A_n-J_n)_{ij} = \widehat\rho^{(n)}_{ij}-1,
\]
are uniformly bounded in $n$.  Hence, for every $p\ge 1$,
\[
\sup_n \frac1{n^2}\sum_{i,j=1}^n \bigl|n(A_n-J_n)_{ij}\bigr|^p < \infty.
\]
Finally, observe that the non-trivial eigenvalues of the symmetric doubly stochastic matrix $A_n$ are exactly the eigenvalues of $B_n$. However, notice that $B_n$ is the matrix of the operator $K_n\vert_{V_n}$ in the orthonormal basis $\{e_{i,n}\}_{i=1}^n$ of $V_n$. Since $K_n$ vanishes on $V_n^{\perp}$, the non-zero eigenvalues of $K_n$ are precisely the eigenvalues of $B_n$. It follows that 
\[
\sup\{ |\lambda| : \lambda \text{ non-trivial eigenvalue of }A_n\}
= \norm{K_n\vert_H}_{L^2\to L^2}
\le \lambda_* + o(1),
\]
where the last inequality uses Lemma~\ref{lem:Knconv}.
Since $\lambda_*<1$, there exists $\delta>0$ such that every non-trivial eigenvalue of $A_n$ has modulus at most $1-\delta$ for all large $n$.
\end{proof}

\subsection{Proof of Proposition~\ref{prop:FinalNail}}
\begin{proof}[Proof of Proposition~\ref{prop:FinalNail}]
The operator $K_n$ is finite-rank, it vanishes on $V_n^{\perp}$, and its restriction to $V_n$ is
represented by the matrix $B_n$. Therefore, the non-zero eigenvalues of $K_n$ are exactly the eigenvalues
of $B_n$, counted with algebraic multiplicity. It follows that
\begin{equation}\label{eq:fred-finite}
\det(I-B_n^2)=\detF(I-K_n^2).
\end{equation}

Next, observe that
\[
K_n^2-K^2=K_n(K_n-K)+(K_n-K)K.
\]
Since the product of two Hilbert-Schmidt operators is trace class,
\[
\norm{K_n^2-K^2}_{\tr}
\le \norm{K_n}_{\HS}\norm{K_n-K}_{\HS}+\norm{K_n-K}_{\HS}\norm{K}_{\HS}
\to 0
\]
by Lemma \ref{lem:Knconv}. Hence $K_n^2\to K^2$ in trace norm. The Fredholm determinant is continuous
with respect to the trace norm~\cite[Chapter~3]{simon2005trace}, so
\begin{equation}\label{eq:fred-conv}
\detF(I-K_n^2)\to \detF(I-K^2).
\end{equation}
Finally, $K\one=0$ and $K|_H=T|_H$. Thus $K^2$ acts as $0$ on the one-dimensional space of constants
and as $(T|_H)^2$ on $H$. Therefore
\begin{equation}\label{eq:fred-ident}
\detF(I-K^2)=\detF\bigl(I-(T|_H)^2\bigr).
\end{equation}
Combining \eqref{eq:fred-finite}, \eqref{eq:fred-conv}, and \eqref{eq:fred-ident} proves the proposition.

\end{proof}

\section*{Acknowledgement}
This problem was discussed in the Open Problem Session in the \emph{Workshop on Mathematical Foundations of Network Models and Their Applications}, as part of the BIRS-CMI pilot
program. The workshop was held at the Chennai Mathematical Institute (CMI) from December 15, 2024, to December 20, 2024. The event was supported by Banff International Research Station (BIRS), Chennai Mathematical Institute (CMI), and the National Board for Higher Mathematics (NBHM). We thank the organizers for a wonderful opportunity.

The part of this work was done during the ICTS-NETWORKS WORKSHOP \emph{Challenges in Networks} from April 06, 2026, to April 17, 2026. We thank the organizers and the ICTS for their hospitality and a stimulating environment.



\providecommand{\bysame}{\leavevmode\hbox to3em{\hrulefill}\thinspace}
\providecommand{\MR}{\relax\ifhmode\unskip\space\fi MR }
\providecommand{\MRhref}[2]{%
  \href{http://www.ams.org/mathscinet-getitem?mr=#1}{#2}
}
\providecommand{\href}[2]{#2}



\end{document}